\documentclass[graybox,envcountsect,envcountsame]{svmult}

\usepackage{mathptmx}       
\usepackage{helvet}         
\usepackage{courier}        
\usepackage{type1cm}        
\usepackage{makeidx}         
\usepackage{graphicx}        
\usepackage{multicol}        
\usepackage[bottom]{footmisc}

\usepackage{amsmath}
\usepackage{amssymb}

\DeclareMathOperator{\M}{M}
\DeclareMathOperator{\A}{I}
\newcommand{\mc}{\mathcal}
\DeclareMathOperator{\Sub}{Sub}
\newcommand{\ul}{\underline}
\usepackage{tikz}

\tikzstyle{dot}=[circle, draw=black!70, fill=black!20, inner sep=.25ex]
\tikzstyle{Dot}=[circle, draw=black, fill=black!70, inner sep=.25ex]
\tikzstyle{line}=[-, draw=black!50, line width=.5pt]
\tikzstyle{thinline}=[-, draw=black!20, line width=.5pt]
\tikzstyle{Line}=[-, draw=black, line width=1pt]

\makeindex             

\usepackage{enumerate}

\begin{document}

\title*{A topos view of the type-2 fuzzy truth value algebra}
\author{John Harding and Carol Walker}
\institute{ \at Dept. of Mathematical Sciences, New Mexico State University, Las Cruces NM 88003, \email{jharding@nmsu.edu, hardy@nmsu.edu}}

\maketitle

\begin{center} {\em Dedicated to Elbert Walker} \end{center}

\abstract*{repeat abstract text here}

\abstract{It is known \cite{Hohle1,Hohle2} that fuzzy set theory can be viewed as taking place within a topos. There are several equivalent ways to construct this topos, one is as the topos of \'{e}tal\'{e} spaces over the topological space $Y=[0,1)$ with lower topology. In this topos, the fuzzy subsets of a set $X$ are the subobjects of the constant \'{e}tal\'{e} $X\times Y$ where $X$ has the discrete topology. Here we show that the type-2 fuzzy truth value algebra \cite{HWW,Zadeh1,Zadeh2} is isomorphic to the complex algebra formed from the subobjects of the constant relational \'{e}tal\'{e} given by the type-1 fuzzy truth value algebra $\mathfrak{I}=([0,1],\wedge,\vee,\neg,0,1)$. More generally, we show that if $L$ is the lattice of open sets of a topological space $Y$ and $\mathfrak{X}$ is a relational structure, then the convolution algebra $L^\mathfrak{X}$ \cite{convolution} is isomorphic to the complex algebra formed from the subobjects of the constant relational \'{e}tal\'{e} given by  $\mathfrak{X}$ in the topos of \'{e}tal\'{e} spaces over $Y$. 
\keywords{Fuzzy sets, type-2 fuzzy sets, topos, complex algebra, convolution \mbox{algebra}, relational structure.\bigskip \\
\textbf{MSC:} 03E72, 03B52, 18B25, 06E25.}}


\section{Introduction}

The type-2 fuzzy truth value algebra $\M$ was introduced by Zadeh \cite{Zadeh1,Zadeh2}. Using $\A$ for the real unit interval $[0,1]$, the underlying set of $\M$ is the set $\A^{\A}$ of all functions $\alpha:\A\to\A$. This algebra is equipped with binary operations $\sqcap,\sqcup$, a unary operation $\neg$, and constants $\underline{0},\underline{1}$. For instance, $\alpha_1\sqcup\alpha_2$ is the function from $\A$ to itself given by 
\[ (\alpha_1\sqcup\alpha_2)(x)=\bigvee\{\alpha_1(y)\wedge\alpha_2(z)\,\mid\, y\vee z = x\}\]
These operations are {\em convolutions} of the corresponding operations $\wedge,\vee,\neg,0,1$ on $\A$, much like polynomial multiplication is a convolution of multiplication. The \mbox{type-2} fuzzy truth value algebra has been investigated by a number of authors. A thorough account is in the monograph \cite{HWW}. 

The construction of $\M$ is an instance of the more general notion of a {\em convolution algebra}. Let $X$ be a set, $R\subseteq X^{n+1}$ be an $n+1$-ary relation on $X$, and $L$ be a complete lattice. Then there is an $n$-ary operation $f$ on the collection $L^X$ of maps $\alpha:X\to L$ given by 
\[ f(\alpha_1,\ldots,\alpha_n)(x)=\bigvee\{\alpha_1(x_1)\wedge\cdots\wedge\alpha_n(x_n)\,\mid\, (x_1,\ldots,x_n,x)\in R\}\]
A {\em relational structure} $\mathfrak{X}$ is a set with a family of relations on it. For a complete lattice $L$, the convolution algebra $L^\mathfrak{X}$ is the algebra obtained by using each $n+1$-ary relation of $\mathfrak{X}$ to produce an $n$-ary operation on $L^X$. Convolution algebras were introduced, and their properties developed, in \cite{convolution}. The type-2 truth value algebra $\M$ is seen as an instance of this considering $\A$ as a relational structure with $\vee$ being the ternary relation $\{(x,y,z)\,\mid\, x\vee y = z\}$ and so forth. 

For $2$ the two-element lattice, the special case of the convolution algebra $2^\mathfrak{X}$ is a well-known object studied in some form for over 100 years. Viewing the elements of $2^X$ as corresponding to subsets $A$ of $X$, the additional operations of $2^\mathfrak{X}$ are given simply by relational image 
\[R[A_1,\ldots,A_n]=\{x\,\mid\, (x_1,\ldots,x_n,x)\in R\mbox{ for some }x_1\in A_1,\ldots,x_n\in A_n\}\] 
The power set $\mc{P}(X)$ equipped with these operations given by relational image is known as the {\em complex algebra} $\mathfrak{X}^+$ of the relational structure $\mathfrak{X}$. Complex algebras play a key role in algebraic treatments of modal logic \cite{Goldblatt}. The best known example is the complex algebra of a group $\mathfrak{G}$ which consists of a binary operation of multiplication on the subsets of $G$ given by $A\cdot B =\{ab\,\mid\, a\in A,b\in B\}$. 

Connections between the type-2 fuzzy truth value algebra $\M$ and the complex algebra of the relational structure $\mathfrak{I}=(\A,\wedge,\vee,\neg,0,1)$ were seen in \cite{HWW} where it was shown that $\M$ and $\mathfrak{I}^+$ satisfy the same equations. This was extended in \cite{convolution} where it was shown that for any relational structure $\mathfrak{X}$ and any non-trivial complete Heyting algebra $L$, that $L^\mathfrak{X}$ and $\mathfrak{X}^+$ satisfy the same equations. The purpose of this note is to further extend the connection between convolution algebras and complex algebras via the use of topos theory. 

An \'{e}tal\'{e} space over a topological space $Y$ is a topological space $E$ together with a local homeomorphism $\pi:E\to Y$. With suitable morphisms, the category of \'{e}tal\'{e} spaces over $Y$ forms a topos \cite{GoldblattTopos,Johnstone,Maclean}. While fuzzy sets do not form a topos, it is known \cite{Hohle1,Hohle2} that they can be embedded into the topos of \'{e}tal\'{e} spaces over the space $Y=[0,1)$ with the lower topology. Under this embedding, fuzzy subsets of a set $X$ correspond to the collection $\Sub(\hat{X})$ of subobjects of the constant \'{e}tal\'{e} $\hat{X}$.

A topos has finite products, so an \'{e}tal\'{e} space has finite powers. An $n$-ary relation $R$ on an \'{e}tal\'{e} space $E$ is a subobject of the power $E^n$. An $n+1$-ary relation on $E$ gives an $n$-ary operation on the collection $\Sub(E)$ of its subobjects. A {\em relational \'{e}tal\'{e}} $\mathcal{E}$ is an \'{e}tal\'{e} space $E$ with a family of relations on it. For a relational \'{e}tal\'{e} $\mathcal{E}$, we define its {\em complex algebra} $\mathcal{E}^+$ to be the collection $\Sub(E)$ of its subobjects, together with the family of operations induced by the relations of $\mathcal{E}$. Our main result is as follows. 
\vspace{2ex}

\noindent {\bf Theorem. } Let $\mathfrak{X}$ be a relational structure and $Y$ be a topological space with lattice of open sets $L$. Then the convolution algebra $L^{\mathfrak{X}}$ is isomorphic to the complex algebra $\hat{\mathfrak{X}}^+$ of the constant relational \'{e}tal\'{e} given by $\mathfrak{X}$. 
\vspace{2ex}

\noindent {\bf Corollary. } Let $Y=[0,1)$ with the lower topology and $\mathfrak{I}=(\A,\wedge,\vee,\neg,0,1)$. Then the type-2 fuzzy truth value algebra $\M$ is isomorphic to the complex algebra $\hat{\mathfrak{I}}^+$ of the constant relational \'{e}tal\'{e} for $\mathfrak{I}$ in the topos of \'{e}tal\'{e} spaces over $Y$. 
\vspace{2ex}

Each topos has an internal language and logic. Objects in a topos have an external view, as described above, as well as an {\em internal} view. For example, a relational \'{e}tal\'{e} is internally simply a relational structure in the topos of \'{e}tal\'{e} spaces. 

In a topos, for each object $E$ there is an object $\mc{P}(E)$ called its {\em power object} that is internally the power set of $E$. A direct external description of $\mc{P}(E)$ is difficult. A more direct link to the external world is provided by the {\em global sections} of $\mc{P}(E)$ which are morphisms $1\to\mc{P}(E)$ where $1$ is the terminal object of the topos. The set of global elements of $\mc{P}(E)$ is in bijective correspondence to the set $\Sub(E)$ of subobjects of $E$. 
For a relational \'{e}tal\'{e} $\mc{E}$, each of its $n+1$-ary relations $R$ provides a morphism $f_R:\mc{P}(E)^n\to\mc{P}(E)$. With these morphisms, $\mc{P}(E)$ becomes the internal complex algebra of $\mc{E}$. Further, these morphisms $f_R$ induce $n$-ary operations on the set of global sections of $\mc{P}(E)$. 
\vspace{2ex}

\noindent {\bf Theorem. } 
Let $\mathfrak{X}$ be a relational structure and $Y$ be a topological space with lattice of open sets $L$. The convolution algebra $L^\mathfrak{X}$ is isomorphic to the algebra of global sections of the internal complex algebra of the relational structure $\hat{\mathfrak X}$ in the topos of \'{e}tal\'{e} spaces over $Y$. 
\vspace{2ex}

\noindent {\bf Corollary. } 
The type-2 fuzzy truth value algebra $\M$ is isomorphic to the algebra of global sections of the internal complex algebra of the constant relational \'{e}tal\'{e} $\hat{\mathfrak{I}}$ in the topos of \'{e}tal\'{e} spaces over $Y=[0,1)$ with the lower topology. 
\vspace{2ex}

The authors of this note worked with Elbert Walker for a period of over a decade on the type-2 fuzzy truth value algebra and convolution algebra. We are pleased to be able to contribute this note to further progress on these subjects. 

\section{The type-2 fuzzy truth value algebra, and convolution algebras}

Let $\A$ be the unit interval $[0,1]$ of the reals. We consider the power 
\[ \A^{\A} = \{\alpha\mid \alpha:\A\to \A\}\]
We consider $\A$ to be a lattice, and let $\neg$ be the negation on $\A$ given by $\neg x=1-x$. An object central to the study of type-2 fuzzy sets \cite{HWW} is the type-2 fuzzy truth value algebra. This was defined by Zadeh \cite{Zadeh1,Zadeh2} as follows. 

\begin{definition}
\label{type2}
The type-2 fuzzy truth value algebra $\M=(\A^{\A},\sqcap,\sqcup,*,\ul{0},\ul{1})$ has its operations defined as follows. 
\begin{align*}
(\alpha\sqcap\beta)(x) &= \bigvee\{\alpha(y)\wedge\beta(z)\mid y\wedge z = x\}\\
(\alpha\sqcup\beta)(x) &= \bigvee\{\alpha(y)\wedge\beta(z)\mid y\vee z = x\}\\
\neg \alpha(x) &= \bigvee\{\alpha(y)\mid \neg y= x\}\\
\ul0(x) &= \begin{cases} 1 \mbox{ if } x=0\\ 0 \mbox{ else } \end{cases} \\ 
\ul1(x) &= \begin{cases} 1 \mbox{ if } x=1\\ 0 \mbox{ else } \end{cases}
\end{align*}
\end{definition}

These operations are defined as {\em convolutions} of the operations $\wedge,\vee,\neg,0,1$ of the unit interval. In \cite{convolution} this process was extended to what is called the convolution algebra of a relational structure over a complete lattice. Before describing this, we note that the binary operations $\wedge,\vee:\A^2\to\A$ can be regarded as ternary relations on $\A$, where for instance 
\[ \wedge =\{(x,y,z)\mid x\wedge y = z\}\]
Similarly, the unary operation $\neg$ on $\A$ can be viewed as a binary relation on $\A$, and the constants $0,1$ of $\A$ can be viewed as unary relations on $\A$. 

\begin{definition}
An indexing set $J$ is simply a set. A type $\tau$ is a function $\tau:J\to\mathbb{N}$ from an indexing set to the natural numbers. We will usually write $\tau(j)$ as $n_j$ when the type $\tau$ is understood from context. 
\end{definition}

\begin{definition}
 An algebra $\mc{A}=(A,(f_j)_J)$ of type $\tau$ consists of a set $A$ and for each $j\in J$ an $n_j$-ary operation $f_j:A^{n_j}\to A$. A relational structure $\mathfrak{X}=(X,(R_j)_J)$ of type~$\tau$ consists of a set $X$ and for each $j\in J$ an $n_j+1$-ary relation $R_j\subseteq X^{n_j+1}$. 
\end{definition}

We may regard $\mathfrak{I}=(\A,\wedge,\vee,\neg,0,1)$ as an algebra of type $2,2,1,0,0$ with two binary operations, a unary operation, and two nullary operations. Also, considering an $n$-ary operation as an $n+1$-ary relation as discussed above, we can consider $\mathbb{I}$ as a relational structure of the same type $2,2,1,0,0$ with two ternary relations, one binary relation, and two unary relations. 

\begin{definition}
\label{aa}
For $L$ a complete lattice and $\mathfrak{X}=(X,(R_j)_J)$ a relational structure of type $\tau$, for each $j\in J$ define an $n_j$-ary operation on $L^X$ by setting \vspace{-1ex}

\begin{align*}
f_{j}(\alpha_1,\ldots,\alpha_{n_j})(x)&=\bigvee\{\alpha_1(x_1)\wedge\cdots\wedge\alpha_{n_j}(x_{n_j})\mid (x_1,\ldots,x_{n_j},x)\in R_j\}
\end{align*}
\vspace{-1ex}

\noindent The convolution algebra is the algebra $L^{\mathfrak{X}}=(L^X,(f_j)_J)$ of type $\tau$.
\end{definition}

The type-2 fuzzy truth value algebra $\M$ is an example of a convolution algebra. Here $\M=\A^{\mathfrak{X}}$ where $\A$ is the unit interval considered as a complete lattice and $\mathfrak{X}$ is the relational structure $(\A,\wedge,\vee,\neg,0,1)$. Another example of convolution algebras comes from the notion of a complex algebra of a relational structure. If $\mathcal{G}=(G,+,-,0)$ is a group, then considering $\mathcal{G}$ as a relational structure of type 2, 1, 0, its complex algebra is the collection $\mc{P}(G)$ of all subsets of $G$ with operations $A+B$, $-A$ and constant $\ul{0}$ defined in the obvious way. Complex algebras play an important role in modal logic. The general definition is as follows. 

\begin{definition}
For a relational structure $\mathfrak{X}=(X,(R_j)_J)$ of type $\tau$, define for each $j\in J$ an $n_j$-ary operation $h_j$ on the power set $\mathcal{P}(X)$ by 

\[ h_j(A_1,\ldots,A_n)=\{a\mid \mbox{ there are $a_1\in A_1,\ldots,a_n\in A_n$ with }(a_1,\ldots,a_n,a)\in R_j\}\]
\vspace{-1ex}

\noindent The complex algebra is the algebra $\mathfrak{X}^+=(\mathcal{P}(X),(h_j)_J)$ of type $\tau$. 
\end{definition}

In \cite{convolution} a number of results were established about the connection between convolution algebras and complex algebras. For 2 the two-element lattice, we have the following from \cite{convolution}.

\begin{theorem}
For a relational structure $\mathfrak{X}$, the convolution algebra $2^\mathfrak{X}$ and  complex algebra $\mathfrak{X}^+$ are isomorphic. 
\end{theorem}

For the next result, we recall that a complete Heyting algebra is a complete lattice that satisfies the infinite distributive law $x\wedge\bigvee_Iy_i = \bigvee_I x\wedge y_i$. Examples of complete Heyting algebras include the real unit interval $\A$ and the collection $\mc{O}(X)$ of open sets of any topological space $X$. 

\begin{theorem} \cite[theorem.~36]{convolution} 
Let $L$ be a complete Heyting algebra with at least two elements and $\mathfrak{X}$ be a relational structure. Then the convolution algebra $L^\mathfrak{X}$ and complex algebra $\mathfrak{X}^+$ satisfy the same equations.
\end{theorem}

These results point to a link between convolution algebras and complex algebras. To further develop this link, we turn to the subject of topos theory.  Our intent is to provide an account of this that does not require prior expertise in topos theory. We next provide the necessary background for this. 

\section{Topos theory}\label{bn}

Here we provide the basics of the portion of topos theory needed for our results. The subject of topos theory is very substantial. For a wider view of the subject, see \cite{GoldblattTopos,Johnstone,Maclean}. 

\begin{definition}
For a topological space $Y$, let $\mc{O}(Y)$ be the lattice of open sets of $Y$. 
\end{definition}

It is well known that $\mc{O}(Y)$ is a complete Heyting algebra. Arbitrary joins are given by unions, finite meets by intersections, negation $\neg A$ is the interior of the complement of $A$, and implication $A\to B$ is the largest open set whose intersection with $A$ is contained in $B$. 

To outline coming ideas, it is helpful to provide a toy example. For this, we will consider the 3-element topological space $Y=\{a,b,c\}$ where open sets are $\emptyset$, $\{b\}$, $\{a,b\}$, $\{b,c\}$, $\{a,b,c\}$. The situation is shown in Figure~\ref{fig1}. 

\begin{figure}[h]
\begin{tikzpicture}[scale=.8]
\draw [thick] (0,0) ellipse (2.3cm and .5cm);
\draw [thick] (-.75,0) ellipse (1.2cm and .3cm);
\draw [thick] (.75,0) ellipse (1.2cm and .3cm);
\node[Dot] (a) at (-1.5,0) {};
\node[Dot] (b) at (0,0) {};
\node[Dot] (c) at (1.5,0) {};
\node at (-1.5,-1) {$a$};  \node at (0,-1) {$b$};  \node at (1.5,-1) {$c$};

\node[Dot] (d) at (6,-1.25) {}; 
\node[Dot] (e) at (6,-.5) {};
\node[Dot] (f) at (5.25,.25) {};
\node[Dot] (g) at (6.75,.25) {};
\node[Dot] (h) at (6,1) {};
\draw[thick] (d)--(e)--(g)--(h)--(f)--(e);
\end{tikzpicture}
\caption{The topological space $Y$ and its lattice $\mc{O}(Y)$ of open sets}
\label{fig1}
\end{figure}
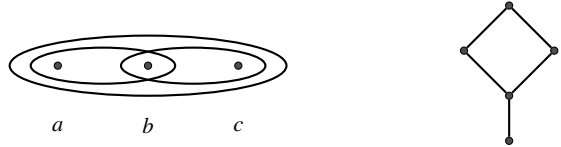

There are a number of equivalent ways to construct a topos from a topological space $Y$, the topos of sheaves over $Y$, the topos of \'{e}tal\'{e} spaces over $Y$, and the topos of $\mc{O}(Y)$-valued sets. These are described in \cite{Hohle1} in detail. Here we shall use the \'{e}tal\'{e} spaces over $Y$ as they are the most concrete and require the least categorical machinery. We give our notation for this in the following, then explain the details. 

\begin{definition}
For $Y$ a topological space, $\mc{E}(Y)$ is the topos of \'{e}tal\'{e} spaces over $Y$. 
\end{definition}

An \'{e}tal\'{e} space over $Y$ is a local homeomorphism $\pi:E\to Y$. This means that $E$ is a topological space and for each $e\in E$ there is an open neighborhood $U$ of $e$ so that the image $\pi(U)$ is open in $Y$ and the restriction $\pi|\,U$ maps $U$ homeomorphically onto its image. When convenient, we denote the \'{e}tal\'{e} space $(E,\pi)$. It is customary to use terminology for this situation. We view an \'{e}tal\'{e} space as a bundle over $Y$ and for each $y\in Y$ call $\pi^{-1}(y)$ the stalk, or fiber, over $y$ and denote this $E_y$. Since $\pi$ is a local homeomorphism, it is clear that the topology of $E$ restricts to the discrete topology on each stalk. 

\begin{definition} 
If $\pi:E\to Y$ and $\pi':E'\to Y$ are \'{e}tal\'{e} space over $Y$, a morphism between them is a continuous map $f:E\to E'$ such that $\pi'\circ f = \pi$.
\end{definition}

A morphism $f:E\to E'$ of \'{e}tal\'{e} spaces takes the stalk $E_y$ over $y$ of one space to the stalk $E'_y$ over $y$ of the other. So a morphism of \'{e}tal\'{e} spaces is a fiberwise morphism that is compatible with the topology. The situation is shown in Figure~\ref{fig2}. With composition of \'{e}tal\'{e} morphisms being usual composition of maps, the category $\mc{E}(Y)$ of \'{e}tal\'{e} spaces over the topological space $Y$ is defined. It is known \cite{GoldblattTopos} that this category forms a topos. Monomorphisms in $\mc{E}(Y)$ are morphisms that are one-one \cite{GoldblattTopos}, or equivalently, are one-one fiberwise.   

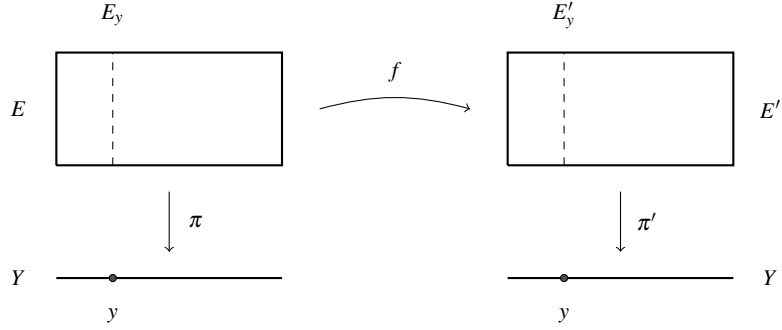
\begin{figure} [h]
\begin{tikzpicture}
\draw[thick] (-1.5,0)--(1.5,0);
\draw[thick] (-1.5,1.5)--(1.5,1.5)--(1.5,3)--(-1.5,3)--(-1.5,1.5);
\draw[thick] (4.5,0)--(7.5,0);
\draw[thick] (4.5,1.5)--(7.5,1.5)--(7.5,3)--(4.5,3)--(4.5,1.5);
\draw[dashed] (-.75,1.5)--(-.75,3);
\draw[dashed] (5.25,1.5)--(5.25,3);
\node[Dot] at (-.75,0) {};
\node[Dot] at (5.25,0) {};
\draw[->] (2,2.25) to [out=15,in=165] (4,2.25);
\draw[->] (0,1.15)--(0,.35);
\draw[->] (6,1.15)--(6,.35);
\node at (-.75,-.5) {$y$};
\node at (5.25,-.5) {$y$};
\node at (-.75,3.5) {$E_y$};
\node at (5.25,3.5) {$E'_y$};
\node at (.35,.75) {$\pi$};
\node at (6.35,.75) {$\pi'$};
\node at (-2,0) {$Y$};
\node at (-2,2.25) {$E$};
\node at (8,0) {$Y$};
\node at (8,2.25) {$E'$};
\node at (3,2.75) {$f$};

\end{tikzpicture}
\caption{A morphism of \'{e}tal\'{e} spaces}
\label{fig2}
\end{figure}

\pagebreak[3]
For an object $A$ in any category, we quasi-order the monomorphisms $e:B\to A$ into $A$ as follows. If $e:B\to A$ and $e':B'\to A$ are two monomorphisms, set $e'\leq e$ if there is a morphism $u:B'\to B$ with $e\circ u = e'$. Let $\equiv$ be the equivalence relation and $\subseteq$ be the partial ordering associated with the quasi-order $\leq$. Basic properties of monomorphisms show that two monomorphisms are equivalent under $\equiv$ if and only if there is an isomorphism $u:B'\to B$ with $e\circ u = e'$.


\begin{definition}
For $A$ an object in a category $\mc{C}$, let $\Sub(A)$ be the collection of equivalence classes of monomorphisms $e:B\to A$ with codomain $A$, partially ordered by $\subseteq$. 
\end{definition}

In any category, $\Sub(A)$ is a poset (modulo cardinality considerations), and in any topos it is a Heyting algebra \cite{Maclean}. In the special case of a topos of \'{e}tal\'{e} spaces, we provide a  proof of this result since it is illustrative of the basic ideas. The key ingredient of the proof is in \cite[Lem.~C.1.3.2(iii)]{Elephant}. 

\begin{proposition}
\label{io}
Let $Y$ be a topological space and $(E,\pi)$ be an \'{e}tal\'{e} space over $Y$. Then $\Sub(E,\pi)$ is isomorphic to the collection $\mc{O}(E)$ of open subsets of $E$, hence is a complete Heyting algebra.  
\end{proposition}

\begin{proof}
Let $E'$ be an open subset of $E$. It is clear that $\pi\,|\,E'$ is a local homeomorphism and the identical embedding $\iota_{E'}:E'\to E$ is an \'{e}tal\'{e} morphism. So the equivalence class $\iota_{E'}/\equiv$ is a subobject of $(E,\pi)$. If $E''$ is another open subset of $E$ and the identical embedding $\iota_{E''}$ belongs to this equivalence class, then there is an isomorphism $u:E''\to E'$ with $\iota_{E'}\circ u = \iota_{E''}$. Since $\iota_{E'}$ and $\iota_{E''}$ are identical embeddings, this implies that $E'$ equals $E''$. So $\iota(\,\cdot\,)/\equiv$ gives an embedding of $\mc{O}(E)$ into $\Sub(E,\pi)$ that is easily seen to be an order-embedding. 



To see this mapping is onto, suppose $\pi':F\to Y$ is a local homeomorphism and $e:F\to E$ is an \'{e}tal\'{e} monomorphism. By \cite[Lem.~C.1.3.2(iii)]{Elephant} a morphism of \'{e}tal\'{e} spaces is in fact a local homeomorphism, hence an open map. So $E'=e(F)$ is an open subset of $E$, so yields an \'{e}tal\'{e} space and monomorphism $\iota_{E'}:E'\to E$. Considering $e':F\to E'$ to be the set map $e$ but with codomain $E'$ we have $e'$ is an \'{e}tal\'{e} isomorphism, and $\iota_{E'}\circ e'=e$. Thus $\iota_{E'}$ belongs to the equivalence class of $e$. So each subobject of $E$ contains a member $\iota_{E'}$ given by an open subset of $E$. 
\end{proof}

We turn next to special \'{e}tal\'{e} spaces that will play a prominent role for us. 

\begin{definition}
For a set $X$, let $\hat{X}$ be the topological product $X\times Y$ of $X$ with the discrete topology and $Y$. Then the projection onto the second component $\pi:\hat{X}\to Y$ is a local homeomorphism. We call $\hat{X}$ the constant \'{e}tal\'{e} space for $X$. 
\end{definition}

\begin{figure}[h]
\begin{tikzpicture}[scale=.58]
\draw [thick] (0,0) ellipse (2.3cm and .5cm);
\draw [thick] (-.75,0) ellipse (1.2cm and .3cm);
\draw [thick] (.75,0) ellipse (1.2cm and .3cm);
\node[Dot] (a) at (-1.5,0) {};
\node[Dot] (b) at (0,0) {};
\node[Dot] (c) at (1.5,0) {};
\node at (-1.5,-1) {$a$};  \node at (0,-1) {$b$};  \node at (1.5,-1) {$c$};

\draw [thick] (0,4.5) ellipse (2.3cm and .5cm);
\draw [thick] (-.75,4.5) ellipse (1.2cm and .3cm);
\draw [thick] (.75,4.5) ellipse (1.2cm and .3cm);
\node[Dot] (a) at (-1.5,4.5) {};
\node[Dot] (b) at (0,4.5) {};
\node[Dot] (c) at (1.5,4.5) {};

\draw [thick] (0,3) ellipse (2.3cm and .5cm);
\draw [thick] (-.75,3) ellipse (1.2cm and .3cm);
\draw [thick] (.75,3) ellipse (1.2cm and .3cm);
\node[Dot] (a) at (-1.5,3) {};
\node[Dot] (b) at (0,3) {};
\node[Dot] (c) at (1.5,3) {};

\draw[thick] (-2.5,2.25)--(2.5,2.25)--(2.5,5.25)--(-2.5,5.25)--(-2.5,2.25);

\node at (-4,3) {$\{p\}\times Y$};
\node at (-4,4.5) {$\{q\}\times Y$};
\node at (-3.1,0) {$Y$};

\end{tikzpicture}
\caption{The constant \'{e}tal\'{e} space for $X=\{p,q\}$ over the space $Y$ of Figure~\ref{fig1}}
\label{fig3}
\end{figure}
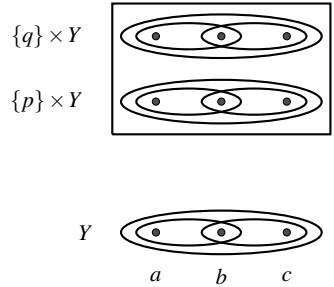

\pagebreak[3]

We consider the matter of subobjects of a constant \'{e}tal\'{e} space $\hat{X}$. In the following, $X\times Y$ is the product of $X$ with the discrete topology and $Y$, and $\pi:X\times Y\to Y$ is the projection onto the second coordinate. For each $x\in X$, we have that $\{x\}$ is open in $X$, so the cross section $\{x\}\times Y$ is open in $X\times Y$, and is homeomorphic to $Y$ in the subspace topology. For $E$ a subset of $X\times Y$, define for $x\in X$ the cross section $E^x$ by 
\[ E^x = (\{x\}\times Y)\cap E\]
Since $E$ is the union of its cross sections, it follows that $E$ is open if and only if the cross section $E^x$ is open for each $x\in X$. Since $\{x\}\times Y$ is homeomorphic to $Y$, the open sets that can occur as $E^x$ for given $x\in X$ are exactly the sets of the form $\{x\}\times A$ where $A$ is open in $Y$. This establishes the following. 

\begin{proposition}
\label{jk}
For a topological space $Y$ and set $X$, there is a bijection between the open subsets of the constant \'{e}tal\'{e} space $\hat{X}$ and indexed families $(A_x)_X$ of open subsets of $Y$. 
\end{proposition}

Combining this with Proposition~\ref{io} and making some trivial observation about maps being order-preserving, we have the following. 

\begin{corollary}
\label{kl}
There is an order-isomorphism $\Phi:\mc{O}(Y)^X\to \Sub(\hat{X})$ where $\Phi(\alpha)$ is the subobject corresponding to the open set $E\subseteq X\times Y$ whose cross section at $x$ is $E^x=\{x\}\times \alpha(x)$. 
\end{corollary}

H\"ohle \cite{Hohle1,Hohle2} has provided a link between fuzzy sets and topos theory. From the perspective of \'{e}tal\'{e} spaces it goes as follows \cite[p.~1179]{Hohle2}. Let $Y$ be the topological space whose underlying set is $[0,1)$ and whose open sets are the initial segments $[0,\lambda)$ for $0\leq \lambda\leq 1$. Then the Heyting algebra $\mc{O}(Y)$ of open sets of $Y$ is isomorphic to the real unit interval $\A=[0,1]$. For a set $X$, the traditional definition of a fuzzy subset of $X$ is a map $\alpha:X\to [0,1]$. By Corollary~\ref{kl} these correspond to open subsets of the constant \'{e}tal\'{e} $\hat{X}$ over $Y$, and hence to subobjects of this constant \'{e}tal\'{e}. To describe the situation, consider Figure~\ref{fig4}. 

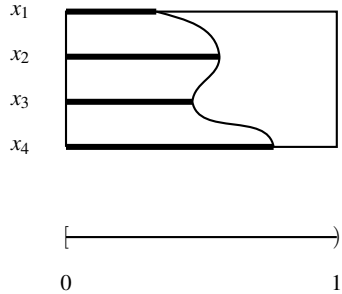
\begin{figure}[h]
\begin{tikzpicture}[scale=1.2]
\draw [thick] (0,0) -- (3,0);
\node at (0,0) {$[$};
\node at (3,0) {$)$};
\draw[thick] (0,1)--(3,1)--(3,2.5)--(0,2.5)--(0,1);
\node at (-.5,1) {$x_4$};
\node at (-.5,1.5) {$x_3$};
\node at (-.5,2) {$x_2$};
\node at (-.5,2.5) {$x_1$};
\node at (0,-.5) {$0$};
\node at (3,-.5) {$1$};
\draw[thick] (1,2.5) to [out=-15,in=95] (1.7,2) to [out=-90,in=85] (1.4,1.5) to [out=-80,in=100] (2.3,1);
\draw[line width=.8mm] (0,2.5)--(1,2.5);
\draw[line width=.8mm] (0,2)--(1.7,2);
\draw[line width=.8mm] (0,1.5)--(1.4,1.5);
\draw[line width=.8mm] (0,1)--(2.3,1);

\end{tikzpicture}
\caption{A subobject of the constant \'{e}tal\'{e} space for $X$ over the space $Y=[0,1)$}
\label{fig4}
\end{figure}

In Figure~\ref{fig4} we show an open subset of the constant \'{e}tal\'{e} $\hat{X}$ for a set $X=\{x_1,\ldots,x_4\}$ over the space $Y=[0,1)$ whose open sets are the initial segments $[0,\lambda)$ where $0\leq\lambda\leq 1$. By Proposition~\ref{jk} this is done by specifying an open subset of $Y$ for each $x\in X$, hence an initial segment $[0,\lambda_x)$ for each $x\in X$. These are the dark bars in the figure. From this open set, we get a function $\alpha:X\to [0,1]$ where $\alpha(x)=\lambda_x$. This is the traditional definition of a fuzzy subset of $X$. In effect, the traditional view of a fuzzy subset of $X$ is obtained by simply turning Figure~\ref{fig4} 90$^0$ counterclockwise!

\section{The convolution algebra from a topos perspective}

Consider the convolution algebra $L^\mathfrak{X}$ in the case that $L$ is the Heyting algebra $\mc{O}(Y)$ of open sets of a topological space $Y$ and $\mathfrak{X}=(X,(R_j)_J)$ is a relational structure of type $\tau:J\to\mathbb{N}$. This convolution algebra is a Heyting algebra with an additional $n_j$-ary operation $f_j$ for each $n_j+1$-ary relation $R_j$. The Heyting algebra underlying this convolution algebra is $L^X$, and we have seen in Corollary~\ref{kl} that this is isomorphic to the Heyting algebra of subobjects of the constant \'{e}tal\'{e} space $\hat{X}$. Here we consider the matter of incorporating additional operations to this setting. 

\begin{definition}
Let $(E_1,\pi_1),\ldots, (E_n,\pi_n)$ be \'{e}tal\'{e} spaces. Their product is the \'{e}tal\'{e} space $(E,\pi)$ where $E$ is the subspace of the product space $E_1\times\cdots\times E_n$  $$E = \{(x_1,\ldots,x_n)\,\mid\, \pi_1(x_1),\ldots,\pi_n(x_n)\mbox{ have the same value}\}$$ and $\pi:E\to Y$ sends $(x_1,\ldots,x_n)$ to the common value of $\pi_i(x_i)$. 
\end{definition}

For an \'{e}tal\'{e} space $(E,\pi)$ we use $(E,\pi)^n$ for the $n$-fold product of $(E,\pi)$ with itself. Recall that an $n$-ary relation on a set $X$ is a subset $R\subseteq X^n$. The corresponding notion for \'{e}tal\'{e} spaces leads to the following. 

\begin{definition}
For $(E,\pi)$ an \'{e}tal\'{e} space, an $n$-ary relation on $(E,\pi)$ is a subobject $R$ of $(E,\pi)^n$. For a type $\tau:J\to\mathbb{N}$, a relational \'{e}tal\'{e} $\mathcal{E}=(E,\pi,(R_j)_J)$ of type $\tau$ consists of an \'{e}tal\'{e} space $(E,\pi)$ and an $n_j+1$-ary relation on $(E,\pi)$ for each $j\in J$. 
\end{definition}

For a set $X$, it is not difficult to see that the $n^{th}$ power $\hat{X}^n$ of the constant \'{e}tal\'{e} $\hat{X}$ is the constant \'{e}tal\'{e} for $X^n$. Thus if $R\subseteq X^n$ is an $n$-ary relation on $X$, then the constant \'{e}tal\'{e} $\hat{R}$ is a subobject of the constant \'{e}tal\'{e} $\hat{X}^n$. We use this in the following. 

\begin{definition}
For $\mathfrak{X}=(X,(R_j)_J)$ a relational structure of type $\tau$, the associated constant relational \'{e}tal\'{e} of type $\tau$ is defined to be $\hat{\mathfrak{X}}=(\hat{X},(\hat{R}_j)_J)$. 
\end{definition}

For an $n+1$-ary relation $R$ on a set $X$ and subsets $A_1,\ldots,A_n\subseteq X$, the relational image $R(A_1,\ldots,A_n)$ is a subset of $X$ that can be constructed in two equivalent ways. Directly, it is given by 
\[R(A_1,\ldots,A_n) = \{x\,\mid\, \mbox{there are $a_1\in A_1,\ldots,a_n\in A_n$ with }(a_1,\ldots,a_n,x)\in R\}\]
This is described algebraically as $R(A_1,\ldots,A_n) = \pi_{n+1}\left((A_1\times\cdots\times A_n\times X)\cap R\right)$ where $\pi_{n+1}$ is projection onto the $(n+1)^{st}$ coordinate. Both approaches can be lifted to the setting of \'{e}tal\'{e} spaces, and give equivalent results.  

\begin{definition}
Let $(E,\pi)$ be a \'{e}tal\'{e} space and $R$ be a subobject of $(E,\pi)^{n+1}$. Then for subobjects $A_1,\ldots,A_n$ of $(E,\pi)$, the \'{e}tal\'{e} relational image is the subobject of $(E,\pi)$ given by 
\[ R(A_1,\ldots,A_n)=\pi_{n+1}((A_1\times\cdots\times A_n\times E)\cap R)\]
\end{definition}

That this gives a subobject follows from the fact that subobjects are given by open subsets, and projection maps are open. It is not difficult to see that the relational image can be given in more direct terms. For each $y\in Y$, the fiber $R_y$ is a subset of $E_y^{n+1}$, hence is an $n+1$-ary relation on $E_y$. Then the image $R(A_1,\ldots,A_n)$ is the subobject of $E$ whose fiber over $y$ is the relational image $R_y((A_1)_y,\ldots,(A_n)_y)$ of the fibers of the subobjects $A_1,\ldots,A_n$. 

\begin{definition}
\label{sd}
For $\mathcal{E}=(E,\pi,(R_j)_J)$ a relational \'{e}tal\'{e} of type $\tau$, use relational \'{e}tal\'{e} image to define its complex algebra $\mathcal{E}^+=(\Sub(E,\pi),(g_j)_J)$ to be the algebra of type $\tau$ where 
\[ g_j(A_1,\ldots,A_{n_j}) = R_j(A_1,\ldots,A_{n_j})\]
\end{definition}

We next consider relational image more closely in the setting of a constant relational \'{e}tal\'{e}. Consider the following example where $X=\{x_1,x_2,x_3,x_4\}$ and $R$ is the ternary relation on $X$ given by $R=\{(x_1,x_1,x_1),(x_2,x_2,x_3),(x_1,x_3,x_4),(x_3,x_2,x_4)\}$. Recall that a subobject $A$ of the constant \'{e}tal\'{e} $\hat{X}$ is given by a specification of an open set for its cross section $A^x$ for each $x\in X$. Two subobjects $A_1$ and $A_2$ of $\hat{X}$ are shown to the left in Figure~\ref{fig5}. For each $x\in X$ the open sets for their cross sections are depicted by thick horizontal bars. The relational image $A=\hat{R}(A_1,A_2)$ is shown at right in this figure in a similar fashion.
\vspace{2ex}

\begin{figure}[h]
\begin{tikzpicture}[xscale=.9,yscale=1.2]
\draw [thick] (0,0) -- (3,0);
\draw[thick] (0,1)--(3,1)--(3,2.5)--(0,2.5)--(0,1);
\node at (-.5,1) {$x_4$};
\node at (-.5,1.5) {$x_3$};
\node at (-.5,2) {$x_2$};
\node at (-.5,2.5) {$x_1$};
\node at (-.5,0) {$Y$};
\draw[line width=.8mm] (0,2.5)--(2,2.5);
\draw[line width=.8mm] (0,2)--(2,2);
\draw[line width=.8mm] (1,1.5)--(3,1.5);
\draw[line width=.8mm] (0,1)--(3,1);
\end{tikzpicture}
\hspace{4ex}
\begin{tikzpicture}[xscale=.9,yscale=1.2]
\draw [thick] (0,0) -- (3,0);
\draw[thick] (0,1)--(3,1)--(3,2.5)--(0,2.5)--(0,1);
\node at (-.5,1) {$x_4$};
\node at (-.5,1.5) {$x_3$};
\node at (-.5,2) {$x_2$};
\node at (-.5,2.5) {$x_1$};
\node at (-.5,0) {$Y$};
\draw[line width=.8mm] (1,2.5)--(3,2.5);
\draw[line width=.8mm] (2,2)--(3,2);
\draw[line width=.8mm] (0,1.5)--(1,1.5);
\draw[line width=.8mm] (0,1)--(2,1);
\end{tikzpicture}
\hspace{4ex}
\begin{tikzpicture}[xscale=.9,yscale=1.2]
\draw [thick] (0,0) -- (3,0);
\draw[thick] (0,1)--(3,1)--(3,2.5)--(0,2.5)--(0,1);
\node at (-.5,1) {$x_4$};
\node at (-.5,1.5) {$x_3$};
\node at (-.5,2) {$x_2$};
\node at (-.5,2.5) {$x_1$};
\node at (-.5,0) {$Y$};
\draw[line width=.8mm] (1,2.5)--(2,2.5);
\draw[line width=.8mm] (0,1)--(1,1);
\draw[line width=.8mm] (2,1)--(3,1);
\end{tikzpicture}

\caption{Subobjects $A_1, A_2$ and their relational image $\hat{R}(A_1,A_2)$ in $\hat{X}$.}
\label{fig5}
\end{figure}
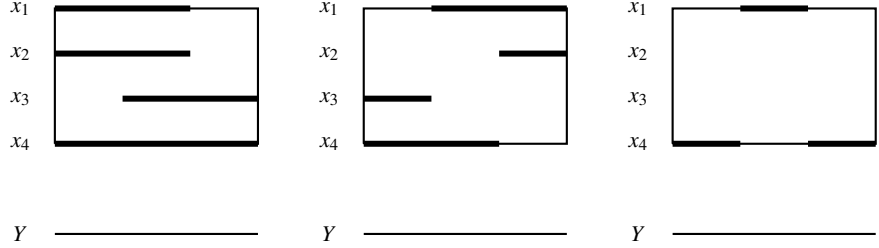

To compute the relational image, we introduce some notation. For each $i=1,2$ and $x\in X$ let $\alpha_i(x)$ and $\alpha(x)$ be the open subsets of $Y$ with $A_i^x=\{x\}\times \alpha_i(x)$ and $A^{x}=\{x\}\times \alpha(x)$. In Figure~\ref{fig5}, $\alpha_1(x_1)$ and $\alpha_1(x_2)$ are the first two thirds of $Y$, $\alpha_1(x_3)$ is the final two thirds of $Y$, and $\alpha_1(x_4)$ is all of $Y$; while $\alpha_2(x_1)$ through $\alpha_2(x_4)$ are the final two thirds, the final third, the first third, and the first two thirds of $Y$ respectively. 

To compute the relational image $A$, we recall that this is done fiberwise. We first focus on the cross section $A^{x_1}=(\{x_1\}\times Y)\cap A$. There is only one way to obtain $x_1$ as the third component of an ordered triple in $R$, that is in the triple $(x_1,x_1,x_1)$. So $(x_1,y)$ is an element of $A^{x_1}$ if and only if $(x_1,y)\in A_1^{x_1}$ and $(x_1,y)\in A_2^{x_1}$. Thus $\alpha(x_1)=\alpha_1(x_1)\cap \alpha_2(x_2)$ is the middle third of~$Y$. There is no way to obtain $x_2$ as the third component of a triple in $R$, so $\alpha(x_2)$ is empty. The only way to obtain $x_3$ as the third component of a triple is in $(x_2,x_2,x_3)$. So $(x_3,y)\in A^{x_3}$ if and only if $(x_2,y)\in A_1^{x_2}$ and $(x_2,y)\in A_2^{x_2}$. So $\alpha(x_3)=\alpha_1(x_2)\cap \alpha_2(x_2)$ which is empty. There are two ways to obtain $x_4$ as the third coordinate of a triple in $R$, from $(x_1,x_3,x_4)$ and from $(x_3,x_2,x_4)$. So $(x_4,y)\in A^{x_4}$ if and only if $(x_1,y)\in A_1^{x_1}$ and $(x_3,y)\in A_2^{x_3}$, or $(x_3,y)\in A_1^{x_3}$ and $(x_2,y)\in A_2^{x_2}$. So $\alpha(x_4)=(\alpha_1(x_1)\cap \alpha_2(x_3))\cup(\alpha_1(x_3)\cap \alpha_2(x_2))$. This is the first and last third of~$Y$. 

\begin{lemma}
\label{mm}
Let $R$ be an $n+1$-ary relation on a set $X$. Let $A_1,\ldots,A_n$ be subobjects of the constant \'{e}tal\'{e} $\hat{X}$ whose cross sections are given by $A_i^x = \{x\}\times\alpha_i(x)$ for $i=1,\ldots, n$. Then the relational image $A=\hat{R}(A_1,\ldots,A_n)$ is the subobject whose cross sections are $A^x=\{x\}\times\alpha(x)$ where 
\[\alpha(x)=\bigcup\{\alpha_1(x_1)\cap\cdots\cap\alpha_n(x_n)\,\mid\, (x_1,\ldots,x_n,x)\in R\}\]
\end{lemma}

\begin{proof}
We have $(x,y)\in A$ if and only if there is $(x_1,\ldots,x_n,x)\in\hat{R}$ with $(x_1,y)\in A_1,\ldots, (x_n,y)\in A_n$. 
\end{proof}


\begin{theorem}
Let $Y$ be a topological space and $L=\mc{O}(Y)$ be its Heyting algebra of open sets. For $\mathfrak{X}=(X,(R_j)_J)$ a relational structure of type $\tau$, there is an isomorphism $\Phi:L^{\mathfrak{X}}\to\hat{\mathfrak{X}}^+$ from the convolution algebra of $\mathfrak{X}$ over $L$ to the complex algebra of the constant relational \'{e}tal\'{e} $\hat{\mathfrak{X}}$ where $\Phi(\alpha)$ is the subobject $A$ of $\hat{X}$ with cross sections $A^x=\{x\}\times\alpha(x)$ for each $x\in X$. 
\end{theorem}

\begin{proof}
Corollary~\ref{kl} shows that $\Phi$ is a Heyting algebra isomorphism. It remains to show that $\Phi$ is compatible with the additional operations of these structures. Let $j\in J$, and $\alpha_1,\ldots,\alpha_{n_j}\in L^X$, and for $i\leq n_j$ let $A_i=\Phi(\alpha_i)$ be the corresponding subobjects of $\hat{X}$. To show that $\Phi$ is compatible with the operations, we must show $\Phi(f_j(\alpha_1,\ldots,\alpha_{n_j})) = \hat{R}_j(A_1,\ldots,A_{n_j})$. 

To see that these subobjects are equal, it is enough to show that for each $x\in X$ their cross sections at $x$ agree. By Definition~\ref{aa} and the definition of $\Phi$, the cross section of $\Phi(f_j(\alpha_1,\ldots,\alpha_{n_j}))$ at $x$ is the open set $\{x\}\times U$ where 

\[ U = \bigvee\{\alpha_1(x_1)\wedge\cdots\wedge\alpha_{n_j}(x_{n_j})\,\mid\, (x_1,\ldots,x_{n_j},x)\in R\}\]

\noindent On the other hand, by Lemma~\ref{mm} the cross section of $\hat{R}(A_1,\ldots,A_{n_j})$ at $x$ is the open set $\{x\}\times V$ where 
\[V=\bigcup\{\alpha_1(x_1)\cap\cdots\cap\alpha_n(x_n)\,\mid\, (x_1,\ldots,x_n,x)\in R\}\]

\noindent Noting that in $L=\mc{O}(Y)$ arbitrary joins are given by unions, and finite meets are given by intersections, we have that $U=V$, hence the cross sections at $x$ agree. 
\end{proof}

In the following, $Y=[0,1)$ is the topological space with open sets the sets $[0,\lambda)$ where $\lambda\leq 1$, $L$ is the lattice $\mc{O}(Y)$ of open sets of $Y$, and $\mathfrak{I} = ([0,1],\wedge,\vee,\neg,0,1)$ is the relational structure built from the unit interval and its operations. 

\begin{corollary}\label{fg}
The type-2 fuzzy truth value algebra $\M$ (see Definition~\ref{type2}) is isomorphic to the convolution algebra $L^\mathfrak{I}$ and is isomorphic to the complex algebra $\hat{\mathfrak{I}}^+$ of the constant relational \'{e}tal\'{e} $\hat{\mathfrak{I}}$ in the topos of \'{e}tal\'{e} spaces over $Y=[0,1)$. 
\end{corollary}

Each topos has an {\em internal language} and an {\em internal logic} \cite{GoldblattTopos,Johnstone,Maclean}. Objects in a topos have an external view which we have described, as well as an internal view given by this internal language. For example, a relational \'{e}tal\'{e} is internally simply a relational structure in the topos of \'{e}tal\'{e} spaces. We wish to use this internal view to describe the complex algebra $\mc{E}^+$ of a relational \'{e}tal\'{e} $\mc{E}$, bearing in mind that the description we have given in Definition~\ref{sd} is an external one in the world of real sets and operations. 

Each object $E$ in a topos has a power object $\mc{P}(E)$ that is the internal version of the power set of $E$. For $E$ an \'{e}tal\'{e} space, the external view of $\mc{P}(E)$ is an \'{e}tal\'{e} space, but one that is highly complex. There is a different way to connect $\mc{P}(E)$ to the external world of sets that is more useful. To describe this, we note that a {\em global section} of an \'{e}tal\'{e} space $(E,\pi)$ is a continuous map $s:Y\to E$ with $\pi\circ s = id_Y$. The following is well known \cite{GoldblattTopos,Elephant,Maclean}.  

\begin{theorem}
Let $E$ be an object in the topos of \'{e}tal\'{e} spaces over $Y$. Then the collection $\mc{G}(\mc{P}(E))$ of global sections of the power object of $E$ forms a Heyting algebra and there is a Heyting algebra isomorphism $\Psi_E:\Sub(E)\to\mc{G}(\mc{P}(E))$. 
\end{theorem}

From general properties of power objects in a topos, if $R$ is an $n+1$-ary relation on an \'{e}tal\'{e} $E$, there is a morphism $h_R:\mc{P}(E)^n\to\mc{P}(E)$ called the relational image morphism. 

\begin{definition}
For a relational \'{e}tal\'{e} $\mc{E}=(E,\pi,(R_j)_J)$, let $\mc{E}^*=(\mc{P}(E),(h_j)_J)$. 
\end{definition}

A relational \'{e}tal\'{e} $\mc{E}$ is internally a relational structure in the topos, and $\mc{E}^*$ is internally its complex algebra. An external link to this internal creature is provided by global sections. 
For $s_1,\ldots,s_n:Y\to\mc{P}(E)$ global sections, let $(s_1,\ldots,s_n)$ be the product map. Then the composite $h_R\circ (s_1,\ldots,s_n):Y\to\mc{P}(E)$ is a global section. 

\begin{definition}
For $\mc{E}=(E,\pi,(R_j)_J)$ a relational \'{e}tal\'{e} of type $\tau$, let $\mc{G}(\mc{E}^*)$ be the the algebra of type $\tau$ consisting of the algebra $\mc{G}(\mc{P}(E))$ of global sections of the power object of $E$ and its derived  operations. 
\end{definition}

The following is not difficult from properties of the internal language, but its proof would take us too far afield. 

\begin{theorem}
For a relational structure $\mathfrak{X}$ and topological space $Y$ with lattice of open sets $L$, the convolution algebra $L^\mathfrak{X}$ is isomorphic to the algebra of global sections $\mc{G}(\hat{\mathfrak{X}}^*)$ of the internal complex algebra of the constant relational \'{e}tal\'{e} $\mathfrak{\hat{X}}$.  
\end{theorem}

\begin{corollary}
The type-2 truth fuzzy value algebra $\M$ is isomorphic to the \mbox{algebra} of global sections $\mc{G}(\hat{\mathfrak{I}}^*)$ of the internal complex algebra of the constant relational \'{e}tal\'{e} given by the relational structure $\mathfrak{I}=(\A,\wedge,\vee,\neg,0,1)$ in the topos of \'{e}tal\'{e} spaces over $Y=[0,1)$ with the lower topology. 
\end{corollary}

We have seen {\em how} the type-2 fuzzy truth value algebra $\M$ arises in the topos used to represent fuzzy sets. The question remains as to {\em why}. Here we can offer some comments, but do not have a fully satisfactory explanation.

The topos of \'{e}tal\'{e} spaces over $Y$ has an internal Dedekind real unit interval, and therefore an internal type-1 fuzzy truth value algebra. However, this is not given by the constant relational \'{e}tal\'{e} $\hat{\mathfrak{I}}$ \cite[p.~327]{Maclean}. Closely related is the {\em rational type-1 truth value algebra} $\mathfrak{I}_\mathbb{Q}$ obtained by restricting the real unit interval and its operations to the rationals. Its internal version is the constant relational \'{e}tal\'{e} $\hat{\mathfrak{I}}_\mathbb{Q}$. 

The {\em rational type-2 fuzzy truth value algebra} $\M_{\mathbb{Q}}$ is defined as is $\M$, but using functions $\alpha:[0,1]_\mathbb{Q}\to [0,1]$ rather than functions $\alpha:[0,1]\to[0,1]$. Our results show it is isomorphic to the algebra of global sections of the complex algebra $\hat{\mathfrak{I}}_{\mathbb{Q}}^*$ of the internal rational type-1 truth value algebra. The idea of taking the complex algebra of an algebra of truth values and using this for truth values of aggregates is not new (see \cite{hyperboolean} and the references therin). 

This line of reasoning may give a partial heuristic as to the occurrence of the type-2 fuzzy truth value algebra $\M$ as the global sections of the internal complex algebra $\hat{\mathfrak{I}}^*$, but points to new questions as well. 
\vspace{2ex}

We thank Jeff Egger and Mamuka Jibladze for several conversations in preparing this manuscript.

\end{document}